\documentclass[11pt,reqno]{amsart}
\usepackage{graphicx, amssymb, color,pdfsync}
\usepackage[all,cmtip]{xy}
\numberwithin{equation}{section}
\usepackage{mathrsfs}
\usepackage{amsmath}
\usepackage{tikz}
\usetikzlibrary{matrix,arrows}
\DeclareMathOperator{\End}{{End}}

\begin{document}
\newcommand{\s}{\vspace{0.2cm}}

\newtheorem{theo}{Theorem}
\newtheorem*{theo*}{Theorem}
\newtheorem{prop}{Proposition}
\newtheorem{coro}{Corollary}
\newtheorem{lemm}{Lemma}
\newtheorem{example}{Example}
\theoremstyle{remark}
\newtheorem{rema}{\bf Remark}
\newtheorem{defi}{\bf Definition}
\newcommand{\rac}{{\mathbb{Q}}}
\newcommand{\comp}{{\mathbb{C}}}
\newcommand{\hip}{{\mathbb{H}}}
\newcommand{\cA}{\mathcal{A}}
\newcommand{\cL}{\mathcal{L}}
\newcommand{\cJ}{\mathcal{J}}
\newcommand{\cK}{\mathcal{K}}
\newcommand{\cH}{\mathcal{H}}
\newcommand{\cC}{\mathcal{C}}
\newcommand{\cG}{\mathcal{G}}
\newcommand{\cM}{\mathcal{M}}
\newcommand{\cW}{\mathcal{W}}

\newcommand{\R}{\mathbb R}
\newcommand{\RR}{\mathbb R}
\newcommand{\QQ}{\mathbb Q}
\newcommand{\NN}{\mathbb N}
\newcommand{\ZZ}{\mathbb Z}
\newcommand{\CC}{\mathbb C}
\newcommand{\PP}{\mathbb P}
\newcommand{\FF}{\mathbb F}

\title[A generalisation of Kani-Rosen decomposition theorem]{A generalisation of Kani-Rosen decomposition theorem for Jacobian varieties}
\date{}

\author{Sebasti\'an Reyes-Carocca and Rub\'i E. Rodr\'iguez}
\address{Departamento de Matem\'atica y Estad\'istica, Universidad de La Frontera, Avenida Francisco Salazar 01145, Casilla 54-D, Temuco, Chile.}
\email{sebastian.reyes@ufrontera.cl, rubi.rodriguez@ufrontera.cl}

\thanks{Partially supported by Postdoctoral Fondecyt Grant 3160002, Fondecyt Grant 1141099 and Anillo ACT 1415 PIA-CONICYT Grant}

\begin{abstract} In this short paper we generalise a theorem  due to Kani and Rosen on decomposition of Jacobian varieties of Riemann surfaces with group action. This generalisation extends the set of Jacobians for which it is possible to obtain an isogeny decomposition where all the factors are Jacobians.
\end{abstract}
\maketitle

\section{Introduction and Statement of the Results}

Let $C$ be a compact Riemann surface of genus $g_C \ge 2$ and $JC$ its Jacobian variety. For a group $H$ of automorphisms of $C$ we denote by $|H|$ its order and by $C_H$ the underlying Riemann surface structure of the quotient $C/H.$

 Kani and Rosen in \cite{KR} studied relations among idempotents in the algebra of rational endomorphisms of an arbitrary abelian variety. By means of these relations, in the case of the Jacobian variety of a Riemann surface $C$ with group action, they succeeded in proving a decomposition theorem for $JC$ in which each factor is isogenous to the Jacobian of a quotient of $C.$

 Applications of Kani-Rosen theorem can be found, for example, in \cite{BF}, \cite{CHQ}, \cite{GP}, \cite{Ruben}, \cite{nos}, \cite{RuRu} and \cite{P}. For the sake of explicitness, we exhibit here this result:

\s

{\bf Theorem C {\bf (Kani and Rosen)}}

Let $H_1, \ldots, H_t$ be groups of automorphisms of a Riemann surface $C$.  If
\begin{enumerate}
\item $H_iH_j=H_jH_i$ for all $i \neq j,$
\item the genus of $C_{H_iH_j}$ is zero for all $i \neq j$ and
\item $g_C=\Sigma_{i=1}^tg_{C_{H_i}}$
\end{enumerate}

then
\begin{equation*}
JC \sim JC_{H_1} \times \cdots \times JC_{H_t} \, .
\end{equation*}

The aim of this article is to provide a generalisation of the aforementioned theorem. More precisely, we prove the following result:

\begin{theo} \label{plano}
Let $\{H_1, \ldots, H_t\}$ be an admissible collection of  groups of automorphisms of a Riemann surface $C.$ Then \begin{equation*}
JC \sim JC_{H_1} \times \cdots \times JC_{H_t} \times P
\end{equation*}
for some abelian subvariety $P$ of $JC$.

Furthermore, if $\{H_1, \ldots, H_t\}$ is admissible and $\displaystyle g_C=\Sigma_{i=1}^tg_{C_{H_i}}\,$,  then
$$JC \sim JC_{H_1} \times \cdots \times JC_{H_t} \, .
$$
\end{theo}

A precise definition of what we call {\it admisible} collection of groups of automorphisms will be provided in Subsection \ref{gad}.

Theorem \ref{plano} enlarges the collection of Jacobians $JC$ for which it is possible to obtain an isogeny decomposition in terms of Jacobians of quotients of $C.$ In fact, at the end of this article, we provide examples of admissible collections which do not satisfy each one of the hypothesis of Theorem C.

\s

We anticipate the fact that the hypothesis of Theorem \ref{plano} can be rephrased in a purely algebraic way; see Proposition \ref{equiv}.

\s

Let $H$ be a group of automorphism of $C.$ We denote by $\pi_H : C \to C_H$ the associated regular covering map, and by $$\pi_H^*: JC_H \to JC$$the induced homomorphism between the respective Jacobian varieties. As the set $\pi_H^*(JC_H)$ is an abelian subvariety of $JC$ which is isogenous to $JC_H$, there exists an abelian subvariety $P(C \to C_H)$ of $JC$ such that \begin{equation} \label{prym} JC \sim JC_H \times P(C \to C_H).\end{equation}

The factor $P(C \to C_H)$ is known as the {\it Prym variety} associated to the covering map induced by $H$.

\s

As we shall see later, the collection $\{H_1\}$ is admissible for each group of automorphisms $H_1$ of $C.$ Hence, the trivial case $t=1$ corresponds to the decomposition \eqref{prym} with $H=H_1$.

The case $t=2$ is slightly more challenging. In fact, it is not a difficult task to produce (or find in the literature; see for example the Klein group case in \cite{RR}) examples of $C$ and pairs of groups of automorphisms $H_1$ and $H_2$ of $C$ such that the dimension of $JC_{H_1} \times JC_{H_2}$ exceeds the genus of $C$. Nevertheless, roughly speaking, part of such an {\it excess} can be geometrically identified. More precisely:

\begin{prop} \label{planodos}
Let $H_1$ and $H_2$ be groups of automorphisms of a Riemann surface $C.$ Then \begin{equation*}
JC \times JC_{\langle H_1, H_2 \rangle} \sim JC_{H_1} \times JC_{H_2} \times P
\end{equation*}
for some abelian subvariety $P$ of $JC$.

In particular, if the genus of $C_{\langle H_1, H_2 \rangle}$ is zero and $\displaystyle g_C=g_{C_{H_1}}+g_{C_{H_2}}$,  then
$$
JC \sim JC_{H_1} \times JC_{H_2}\, .
$$
\end{prop}

Note that, in the previous proposition, the collection $\{H_1, H_2\}$ is not asked to be admissible. The key point in the proof of this result is the classically known formula to compute the dimension of the sum of two vector subspaces. The same argument cannot be generalized for $t \ge 3$ since, as a matter of fact, there does not exist such a formula for more than three summands.

\s

As a direct consequence of Theorem \ref{plano} we obtain:

\begin{coro} \label{rubi} Let $\{H_1, \ldots, H_t\}$ be an admissible collection of groups of automorphism of a Riemann surface $C$.

Then, for each fixed $k \in \{1, \ldots, t\}$, the following product
\begin{equation*}
JC_{H_1} \times \cdots \times JC_{H_{k-1}} \times JC_{H_{k+1}} \times \cdots \times JC_{H_t}
\end{equation*}is isogenous to an abelian subvariety of  $P(C \to C_{H_{k}})$.

Furthermore, if $\{H_1, \ldots, H_t\}$ is admissible and $\displaystyle g_C=\Sigma_{i=1}^tg_{C_{H_i}}$,  then
\begin{equation*}
JC_{H_1} \times \cdots \times JC_{H_{k-1}} \times JC_{H_{k+1}} \times \cdots \times JC_{H_t} \sim P(C \to C_{H_{k}})\, .
\end{equation*}
\end{coro}

The previous result can be employed to obtain, in a very straightforward way, Jacobians which are isogenous to Prym varieties (by considering $t=2$).

\s

As a further consequence of Theorem \ref{plano}, we prove the following corollary regarding the existence of Riemann surfaces with prescribed data on their Jacobian varieties.

\begin{coro} \label{fib}
Let $t \ge 2$ and $g_1, \ldots, g_t \ge 1$ be integers. Given $t$ hyperelliptic Riemann surfaces $C_1, \ldots, C_t$ of genus $g_1, \ldots, g_t$ respectively, there exists a Riemann surface $C$ of genus $$g=1-2^t+2^{t-1}(t+\Sigma_{i=1}^t g_{i})$$ such that
$$
JC \sim JC_1 \times \cdots \times JC_t \times P \, ,
$$
for a suitable abelian subvariety $P$ of $JC$ of dimension
$$
1+ 2^{t-1}t-2^t+(2^{t-1}-1)\Sigma_{i=1}^tg_{i} \, .
$$
\end{coro}

\s

In \cite{EJPS} Ekedahl and Serre introduced the problem of determining all genera for which there is a Riemann surface whose Jacobian can be decomposed, up to isogeny, as a product of elliptic curves (for recent progress on this topic we refer to \cite{PA}). Since their paper, there has been much interest in this sort of Riemann surfaces, particularly in their applications to number theory.

More restrictive approaches to this problem were developed by Paulhus in \cite{P} and Hidalgo in \cite{Ruben}. More precisely, Paulhus dealt with the following  problem: given an integer $g,$ she asked for the largest integer $t=t(g)$ such that $E^t$ is isogenous to an abelian subvariety of $JC,$ for some Riemann surface $C$ of genus $g$ and for some elliptic curve $E.$ She answered this question for small genera. Meanwhile, a similar problem was considered by Hidalgo: given an integer $t,$ he asked for the existence of a Riemann surface $C$ of minimum genus $g=g(t)$, such that its Jacobian is isogenous to the product of at least $t$ elliptic curves and other Jacobians. By means of very explicit constructions, he determined bounds for $g$ in terms of $t.$

\s

Let $E_1, \ldots, E_t$ be $t\ge 2$ elliptic curves. We remark that, as a direct application of the previous corollary, we can ensure the existence of a Riemann surface of genus $g_t=1+2^t(t-1)$ so that its Jacobian is sogenous to the product $E_1 \times \cdots \times E_t \times P$ for a suitable abelian subvariety $P$ of it. Furthermore, using our theorem we  are in position to guarantee the existence of a Riemann surface of smaller genus than $g_t$ and satisfying the same property, as shown in the next corollary:

\begin{coro} \label{fib4}
Let $t \ge 2$ be an integer and let $E_1, \ldots, E_{t}$ be elliptic curves. Then there exists a compact Riemann surface $C$ of genus

\begin{displaymath}
g= \left\{ \begin{array}{ll}
 1-2^{t/2}+(3t)2^{(t/2)-2} & \textrm{if $t$ is even}\\
1-2^{(t+1)/2}+(3t+1)2^{(t-3)/2} & \textrm{if $t$ is odd}
  \end{array} \right.
\end{displaymath}so that $$JC \sim E_1 \times \cdots \times E_{t} \times P$$for a suitable abelian subvariety $P$ of $JC$ of dimension $g-t.$

In particular, if the elliptic curves are pairwise isogenous, then $$JC \sim E_1^{t} \times P.$$
\end{coro}

\s

The next theorem was also proved in  \cite{KR}; we give an alternative proof, substantially simpler than the original one.

\s

{\bf Theorem B (Kani and Rosen)} Let $G$ be a finite group of automorphisms of a Riemann surface $C$ such that $G=H_1 \cup \ldots \cup H_t$ where the subgroups $H_i$ of $G$ satisfy $H_i \cap H_j=\{1\}$ for $i \neq j$. Then
\begin{equation*}
JC^{t-1} \times JC_G^{|G|} \sim JC_{H_1}^{|H_1|} \times \cdots \times JC_{H_t}^{|H_t| } \, .
\end{equation*}

\medskip

This paper is organized as follows. In Section \ref{pre} we shall briefly review the preliminaries:  representations of groups and the group algebra decomposition theorem for Jacobians varieties (on which the proofs are based). Section \ref{proofs} is devoted to proving the results. In the last section we exhibit an explicit example in order to illustrate how our results  can be applied.

\s

{\bf Acknowledgments.} The authors are grateful to their colleague Angel Carocca for his helpful  suggestions throughout the preparation of this manuscript.

\section{Preliminaries} \label{pre}
\subsection{Representations of groups} Let $G$ be a finite group and let $\rho : G \to \mbox{GL}(V)$ be a complex representation of $G$. By abuse of notation, we shall also write $V$  to refer to the representation $\rho$. The {\it degree} $d_V$ of $V$ is the dimension of $V$ as a complex vector space, and the {\it character} of $V$ is the map obtained by associating to each $g \in G$ the trace of the matrix $\rho(g)$. Two representations $V_1$ and $V_2$ are {\it equivalent} if and only if their characters agree; we write $V_1 \cong V_2.$ The {\it character field} $K_V$ of $V$ is the field obtained by extending the rational numbers by the character values.  The {\it Schur index} $s_V$ of $V$ is the smallest positive integer such that there exists a field extension of $K_V$ of degree $s_V$ over which $V$ can be defined.

It is known that for each rational irreducible representation $W$ of $G$ there is a complex irreducible representation $V$ of $G$ such that
\begin{equation} \label{cambio}
W \otimes_{\mathbb{Q}}\mathbb{C}  \cong   (\oplus_{\sigma} V^{\sigma}) \oplus \stackrel{s_V}{\cdots}\oplus (\oplus_{\sigma} V^{\sigma}) = s_V\, \left( \oplus_{\sigma} V^{\sigma} \right) \, ,
\end{equation}
where the sum $\oplus_{\sigma}V^{\sigma}$ is taken over the Galois group associated to $\mathbb{Q} \le K_V.$ The representation $V$ is said to be {\it associated} to $W$. If $V'$ is another complex irreducible representation associated to $W$ then $V$ and $V'$ are said to be {\it Galois associated}.

\s

Let $H$ be a subgroup of $G$. We denote by $V^H$ the vector subspace of $V$ consisting of those elements which are fixed under $H.$ By Frobenius Reciprocity theorem, its dimension --denoted by $d_V^H$-- agrees with $\langle \rho_H , V \rangle_G$, where $\rho_H$ stands for the representation of $G$ induced by the trivial one of $H,$ and the brackets for the usual inner product of characters.

We refer to \cite{Serre} for further basic facts related to representations of groups.

\subsection{Complex tori and abelian varieties} A $g$-dimensional {\it complex torus} $X=V/\Lambda$ is the quotient between a $g$-dimensional complex vector space $V$ by a maximal rank discrete subgroup $\Lambda$. Each complex torus is an abelian group and a $g$-dimensional compact connected complex analytic manifold. {\it Homomorphisms} between  complex tori are holomorphic maps which  are homomorphisms of groups; we shall denote by $\End(X)$ the ring of endomorphisms of $X.$ An {\it isogeny} between two complex tori is a surjective homomorphism with finite kernel; isogenous tori are denoted by $X_1 \sim X_2$. The isogenies of a complex torus $X$ into itself are  the invertible elements of the ring of rational endomorphisms of $X$
$$
\End_{\mathbb{Q}}(X):=\End(X) \otimes_{\mathbb{Z}} \mathbb{Q} \, .
$$

An {\it abelian variety} is by definition a complex torus which is also a complex projective algebraic variety. The Jacobian variety $JC$ of a Riemann surface $C$ of genus $g$ is an  abelian variety of dimension $g.$

We refer to \cite{bl} for basic material on this topic.

\subsection{Group algebra decomposition} \label{gad} We consider a finite group $G$ and its rational irreducible representations $W_1, \ldots, W_r$. It is classically  known that if $G$ acts on $C$ then this action  induces a $\mathbb{Q}$-algebra homomorphism $$\Phi : \mathbb{Q} [G]\to \End_{\QQ}(JC).$$

For each $ \alpha \in {\mathbb Q}[G]$ we define the abelian subvariety $$A_{\alpha} := {\textup Im} (\alpha)=\Phi (l\alpha)(JC) \subset JC$$where $l$ is some positive integer chosen in such a way that $l\alpha \in {\mathbb Z}[G]$.

The decomposition of $1 = e_1 + \cdots + e_r \in \mathbb{Q}[G]$, where each $e_l$ is a uniquely determined central idempotent (computed from $W_l$), yields an isogeny $$JC \sim A_{e_1} \times \cdots \times A_{e_r}$$
which is $G$-equivariant. See \cite{l-r}.

Additionally, there are idempotents $f_{l1},\dots, f_{ln_l}$ such that $e_l=f_{l1}+\dots +f_{ln_l}$ where $n_l=d_{V_l}/s_{V_l}$ and $V_l$ is a complex irreducible representation of $G$ associated to $W_l.$ These idempotents provide $n_l$ subvarieties of $JC$ which are isogenous between them; let $B_l$ be one of them, for every $l.$ Thus
\begin{equation} \label{eq:gadec}
JC \sim_G B_{1}^{n_1} \times \cdots \times B_{r}^{n_r} \, ,
\end{equation}
called the {\it group algebra decomposition} of $JC$ with respect to $G$. See \cite{cr}.

If the representations are labeled in such a way that $W_1(=V_1)$ denotes the trivial one (as we will do in this paper) then $n_1=1$ and $B_{1} \sim JC_G$.

\s

{\bf Notation.} Throughout this paper we shall reserve the notation $\sim_G$ to refer to the group algebra decomposition \eqref{eq:gadec} of $JC$ with respect to $G$. Observe that each product
$B_i^{n_i}$ admits a $G$-action (by an appropriate multiple of $W_i$), but, in general, each $B_i$ does not.

\s

Let $H$ be a subgroup of $G$. It was also proved in \cite{cr}  that the group algebra decomposition of $JC$ with respect to $G$ induces a isogeny decomposition
\begin{equation} \label{ind}
JC_H \sim  B_{1}^{{n}_1^H} \times \cdots \times B_{r}^{n_r^H}
\end{equation}of $JC_H,$
where ${n}_l^H=d_{V_l}^H/s_{V_l} \,$.

Note that, in general, the isogeny \eqref{ind} is not a group algebra decomposition because the quotient $C_H$ does not necessarily have group action. We also mention that the dimension of each factor in \eqref{eq:gadec} is explicitly computable in terms of the monodromy of the action of $G$ on $C.$ See \cite{yoibero}.

\s

Now, once the basic preliminaries have been introduced, we are in position to bring in the precise definition of {\it admissible} collection of automorphisms.

\s

\begin{defi}
Let $C$ be a compact Riemann surface and let $G, H_1, \ldots, H_t$ be subgroups of automorphisms of $C$ such that $G$ contains $H_i$ for each $1 \le i \le t$. Consider the group algebra decomposition \eqref{eq:gadec} with respect to $G$.

The collection $\{H_1, \ldots, H_t\}$ will be called {\it $G$-admissible} if
$$d_{V_l}^{H_1}+ \cdots + d_{V_l}^{H_t}  \le d_{V_l}$$for every complex irreducible representation $V_l$ of $G$ such that $B_l \neq 0$. The collection will be called {\it admissible} if it is $G$-admissible for some group $G$.
\end{defi}

We emphasize the fact that our definition of admissibility is based on the dimensions of the vector subspaces fixed under the corresponding subgroups;
consequently, it is based on the induced isogenies \eqref{ind} with $H=H_i$. As the reader will note in the next section,
these isogenies will play a key role in our proofs; indeed, this is the new ingredient that  was not available when the Kani-Rosen decomposition theorem was originally proved.

\section{Proofs} \label{proofs}

\subsection{Proof of Theorem \ref{plano}} Let us assume that the collection $\{H_1, \ldots, H_t\}$ is $G$-admissible, and consider the group algebra decomposition of $JC$ with respect to $G$
given by  \begin{equation}\label{estuche}
JC \sim_G B_1^{n_1} \times \cdots \times B_r^{n_r}
\end{equation}
and
\begin{equation}\label{verde}
JC_{H_{i}} \sim B_1^{{n_1}^{H_i}} \times \cdots \times B_r^{{n_r}^{H_i}}
\end{equation}the corresponding induced isogeny decomposition of $JC_{H_{i}}$.

Suppose that the factor $B_l$ is associated to the rational irreducible representation $W_l$, and that this is, in turn, associated to the complex irreducible representation $V_l$ of $G$.

\s

If $t=1$, the decomposition \eqref{prym} with $H=H_1$ proves the result. Thus, from now on we assume $t>1.$

\s

{\bf Claim.} The genus of $S_G$ is zero.

\s

Assume that the genus of $S_G$ is positive or, equivalently, that the factor $B_1$ has positive dimension. Note that$$d_{V_1}^{H_1}+\cdots + d_{V_1}^{H_t}=t$$and therefore, as the collection $\{H_1, \ldots, H_t\}$ is $G$-admissible, we must have$$t=d_{V_1}^{H_1}+\cdots + d_{V_1}^{H_t} \leq d_{V_1} = 1 \, ;$$a contradiction. This proves the claim.

\s

It follows that the factor $B_1$ equals zero. Without loss of generality, we may assume that $B_l \neq 0$ for each $l > 1$. Now, as the collection $\{H_1, \ldots, H_t\}$ is supposed to be $G$-admissible, we have that
\begin{equation*}  d_{V_l}^{H_1}+\cdots + d_{V_l}^{H_t}+ \delta_l = d_{V_l}
\end{equation*}for some $\delta_l \ge 0,$ for each $l \neq 1$. The last equality can be also written as
\begin{equation*}  n_{l}^{H_1}+\cdots + n_{l}^{H_t}+ \tilde{\delta}_l = n_{l} \, ,
\end{equation*}
where $\tilde{\delta}_l=\delta_l/s_{V_l}$ for each $l \neq 1$. We remark that all $\tilde{\delta}_l$ are integers.

\s

In this way, we obtain that
$$B_l^{n_{l}^{H_1}} \times \cdots \times B_l^{n_{l}^{H_t}} \times B_l^{\tilde{\delta}_l} = B_l^{n_{l}}
$$
for all $l \neq 1.$

Thereby
$$
\prod_{l=1}^r B_l^{\tilde{\delta}_l} \times  \prod_{l=1}^r (B_l^{n_{l}^{H_1}} \times \cdots \times B_l^{n_{l}^{H_t}} ) = \prod_{l=1}^r B_l^{n_l}$$or, equivalently, if we reorder the products, we see that$$
\prod_{l=1}^r B_l^{\tilde{\delta}_l} \times  \prod_{i=1}^t   \prod_{l=1}^r B_l^{n_{l}^{H_i}} = \prod_{l=1}^r B_l^{n_l}$$

Now, by considering the isogenies \eqref{estuche} and \eqref{verde}, it follows that

\begin{equation}\label{corchete}P \times JC_{H_1} \times \cdots \times JC_{H_t}
 \sim JC \, ,
\end{equation}where $P:=\Pi_{l=1}^r B_l^{\tilde{\delta_l}} $. This proves the first result.

\s

The second one is now straightforward. Indeed, if we suppose $\{H_1, \ldots, H_t\}$ to be $G$-admissible and $g_C$ to be equal to $\Sigma_{i=1}^tg_{C_{H_i}},$ then by comparing  dimensions in both sides of \eqref{corchete}, the factor $P$  must clearly be zero. Thereby \begin{equation*}\label{corchete2}JC \sim JC_{H_1} \times \cdots \times JC_{H_t}
\end{equation*}
in this case, and the proof is complete. \qed

\subsection{Restatement of the hypothesis of Theorem \ref{plano}}
Let $H_1(C, \mathbb{Z})$ denote the integral first homology group of $C$. It is classically known that the action of $G$ on $C$ induces a representation of degree $2g$ $$ \rho_{\textup{rac}} : G \to \mbox{GL}(H_1(C, \mathbb{Z}) \otimes_{\mathbb{Z}} \mathbb{Q})$$of $G$, known as the {\it rational} representation of $G.$ By abuse of notation, we shall also write $\rho_{\textup{rac}}$  to refer to its complexification $\rho_{\textup{rac}}\otimes \mathbb{C}$.

\s

The next result exhibits algebraic restatements of the hypothesis of our main theorem:
\s

\begin{prop} \label{equiv}Let $C$ be a compact Riemann surface and let $G, H_1, \ldots, H_t$ be subgroups of automorphisms of $C$ such that $G$ contains $H_i$ for each $1 \le i \le t$. Consider the group algebra decomposition \eqref{eq:gadec} with respect to $G$.

The following statements are equivalent:
\begin{enumerate}
\item the collection $\{H_1, \ldots, H_t\}$ is $G$-admissible and $\displaystyle g_C=\Sigma_{i=1}^tg_{C_{H_i}}$
\item $\Sigma_{i=1}^t d_{V_l}^{H_i}  =d_{V_l}$ for every $l$ such that $B_l \neq 0.$
\item there are non-negative integers $a_l$ such that $$ \oplus_{i=1}^t \rho_{H_i} \cong \sum_{ \langle  \rho_{\textup{rac}}, W_l \rangle_G \neq 0  }n_{l} W_l \oplus \sum_{ \langle  \rho_{\textup{rac}}, W_l \rangle_G =0 }a_l W_l$$

\end{enumerate}
\end{prop}

\begin{proof} The equivalence between statements $(1)$ and $(2)$ follows directly from the proof of Theorem \ref{plano}.

We proceed to verify the one between statements $(2)$ and $(3).$ It is a known fact that $$\rho_{H_i} \cong \oplus_{l=1}^r n_l^{H_i} W_l$$for each $1 \le i \le t$ (see, for example \cite[Lemma 4.3]{cr}). Thus, $$\oplus_{i=1}^t \rho_{H_i} \cong \oplus_{l=1}^r (n_l^{H_1}+ \cdots + n_l^{H_t})W_l.$$Now, as \begin{equation*} \label{perro1}
\dim B_l=0 \ \ \textup{ if and only if } \  \ \langle \rho_{\textup{rac}}, W_l \rangle_G = 0,
\end{equation*} we see that if statement $(2)$ is assumed, then
$$ \oplus_{i=1}^t \rho_{H_i} \cong \sum_{ \langle  \rho_{\textup{rac}}, W_l \rangle_G \neq 0  }n_{l} W_l \oplus \sum_{ \langle  \rho_{\textup{rac}}, W_l \rangle_G =0 }(n_l^{H_1}+ \cdots + n_l^{H_t}) W_l$$proving statement $(3)$ with $a_{l}=n_l^{H_1}+ \cdots + n_l^{H_t}.$ The converse is direct.
\end{proof}

As a particular case, we consider the situation when the dimension of $B_l$ equals zero if and only if $l=1.$ Then, the proposition above ensures that in this case the collection $\{H_1, \ldots, H_t\}$ is $G$-admissible and $\displaystyle g_C=\Sigma_{i=1}^tg_{C_{H_i}}$ if and only if \begin{equation} \label{ocho} \oplus_{i=1}^t \rho_{H_i} \cong \sum_{ l \neq 1  }n_{l} W_l \oplus a_1W_1 \cong \rho_{\textup{reg}} \oplus (a_1-1)W_1 \end{equation}for some $a_1 \ge 1$, where $\rho_{\textup{reg}}$ stands for the regular representation of $G$.

The value of $a_1$ can easily be determined by comparing in \eqref{ocho} either the character at the identity or the multiplicity of the trivial representation; namely
$$a_1=1+|G| \cdot \left(-1+ \sum_{i=1}^t \frac{1}{|H_i|}\right)=t.$$ Hence, at the end, the hypothesis of Theorem \ref{plano} are equivalent to \begin{equation*}
\oplus_{i=1}^t \rho_{H_i}  \cong \rho_{\tiny \textup{reg}} \oplus (t-1) W_1
\end{equation*}
in this case.

\subsection{Proof of Proposition \ref{planodos}} Let $H_1$ and $H_2$ be two subgroups of automorphisms of $C$. Clearly, there is an integer $\delta_l \ge 0$ such that $$\dim (V_l^{H_1}+ V_l^{H_2})+ \delta_l = \dim V_l$$for each $l.$ By considering the dimension formula for the sum of two vector subspaces, we can assert that
$$
\dim V_l^{H_1}+ \dim V_l^{H_2}-\dim (V_l^{H_1} \cap V_l^{H_2})+ \delta_l = \dim V_l \,.
$$

Now, as
$$
V_l^{H_1} \cap V_l^{H_2}=V_l^{\langle H_1, H_2 \rangle} \, ,
$$
the following equality is obtained:$$n_l^{H_1}+n_l^{H_2}+ \tilde{\delta_l}=n_l+n_l^{\langle H_1, H_2 \rangle}$$
where $\tilde{\delta_l}=\delta_l/s_{V_l}$ is a non negative integer.

The remaining part of the proof follows analogously as done in the proof of Theorem \ref{plano}. \qed

\subsection{Proof of Corollary \ref{rubi}} The isogeny decomposition \eqref{prym} of $JC$ associated to the regular covering map $\pi_{H_k} : C \to C_{H_k}$ together with Theorem \ref{plano} ensure that \begin{equation*}
JC_{H_1} \times \cdots \times JC_{H_{k-1}} \times JC_{H_{k+1}} \times \cdots \times JC_{H_t} \times P \sim P(C \to C_{H_{k}}) \, ,
\end{equation*}for some abelian subvariety $P$ of $JC$, and the proof is done.
\qed

\subsection{Proof of Corollary \ref{fib}}  We start by recalling the well-known fact that for each $1 \le i \le t$ there exists a two-fold regular covering map over the Riemann sphere
$$
f_i : C_i \to \mathbb{P}^1
$$
ramified over $2g_i+2$ values. Furthermore, if $B_i \subset \mathbb{P}^1$ denotes the set of branch values of $f_i$, then we  choose, without loss of generality, the $f_i$ such that the intersection $B_i \cap B_j$  is empty for $i \neq j$. We denote by $H_i \cong \mathbb{Z}_2$ the deck group of $ f_i $; that is, $C_i/H_i \cong \mathbb{P}^1$.

Let
$$
C=\{(z_1, \ldots, z_t) \in C_1 \times \cdots \times C_t : f_1(z_1)=\cdots = f_t(z_t)\}
$$
be the fiber product of the coverings $f_1, \ldots, f_t$. Clearly, $C$  is endowed with canonical projections
$$
\pi_i : C \to C_i \hspace{0.5 cm} (z_1, \ldots, z_t) \mapsto z_i \, .
$$

We recall that $C$ has the structure of a compact Riemann surface (or, equivalently, a connected irreducible smooth complex projective algebraic curve; see \cite{F}, \cite{RA} and \cite{RAS}), and that the direct product
$$
H=H_1 \times \cdots \times H_t \cong \mathbb{Z}_2^t
$$
acts canonically  on $C$. It is not difficult to see that the correspondence $$f: C \to \mathbb{P}^1 \ \ , \ \ z=(z_1, \ldots, z_t) \mapsto f(z):=f_1(z_1) = \cdots = f_t(z_t)$$
is a branched regular covering map of degree $2^t$ admitting $H$ as deck group. Note that, as the coverings $f_i$ have been chosen with disjoint ramification, the set of branch values $B$ of $f$ agrees with $B_1 \cup \ldots \cup B_t$.

The projection $\pi_i : C \to C_i$ is a regular covering map of degree $2^{t-1}$ admitting
$$
K_i \cong H/H_i \cong \prod_{k=1, k \neq i}^t H_k
$$
as deck group; then $C_i \cong C/K_i$.

\s

{\bf Claim.} The collection $\{K_1, \ldots, K_t\}$ is $H$-admissible.

\s

The complex irreducible representations of $H \cong \langle K_1 , \ldots, K_t\rangle$ are of the form
$$
V:=V_1 \otimes \cdots \otimes V_t \, ,
$$
where $V_j$ is a complex irreducible representation of $H_j.$ Note that $$V^{K_i} \cong V_{1}^{H_1} \otimes \cdots \otimes V_{i-1}^{H_{i-1}} \otimes V_i \otimes V_{i+1}^{H_{i+1}} \otimes \cdots \otimes V_{t}^{H_t}$$

We remark the following obvious observation: if $V_j$ is the non-trivial representation of $H_j \cong \mathbb{Z}_2$ then $$ V_{j}^{H_j}=0.$$Thus, it follows that:

\begin{enumerate}
\item if the representation $V_j$ is non-trivial for some $j \neq i$, then $$d_{V}^{K_i}=0\, ;$$
\item if the representations $V_j$ are trivial for all $j \neq i$, then $$d_{V}^{K_i}=d_{V_i} = 1\, .$$
\end{enumerate}

Thereby
$$d_{V}^{K_1}+ \cdots + d_{V}^{K_t} \le 1 = d_{V} $$and the proof of the claim is done.
\s

It now follows from Theorem \ref{plano} that $$JC \sim JC_1 \times \cdots \times JC_t \times P$$for an abelian subvariety $P$ of $JC.$ Clearly, the dimension of $P$ is $g_C-\Sigma_{i=1}^tg_{i}.$

\s

It only remains to compute the genus of $C$. This task can be accomplished by applying the Riemann-Hurwitz formula (see, for example \cite[p. 80]{RM}) to the involved coverings. More precisely, this formula applied to $f_i$ says that \begin{equation} \label{rr1} R_i=g_i+1\end{equation}where $R_i$ is the branch number of $f_i.$ We recall that $R_i=\Sigma(1-1/m_l)$ where the sum is taken over a maximal collection of non-equivalent branch points $\{p_1^i, \ldots, p_s^i\}$ of $f_i$ and with $m_l$ denoting the order of the $H_i$-stabilizer subgroup of $p_l.$  Similarly, the Riemann-Hurwitz formula applied to $f$ says that\begin{equation} \label{rr2} 2g_C-2=2^t(-2)+2^tR. \end{equation}

Again, as the ramificacion of the coverings $f_i$ have been chosen to be disjoint, the branch number $R$ of $f$ agrees with $R_1+\cdots+R_t$. Finally,  the desired expression for $g_C$ is obtained after replacing \eqref{rr1} in \eqref{rr2}.

The proof is done. \qed

\subsection{Proof of Corollary \ref{fib4}} It is a known fact that, given two elliptic curves $E_1$ and $E_2,$ there exists a compact Riemann surface $C$ of genus two such that $JC$ is isogenous to the product $E_1 \times E_2$ (see, for example \cite{Earle}, \cite{jap} and \cite{RR2}; see also \cite[Theorem 1]{Ruben} for an explicit construction).

\s

Assume $t$ to be even; say $t=2s$ for some $s \ge 1.$ Let $C_j$ be a Riemann surface such that $$JC_j \sim E_j \times E_{j+s}$$for each $1 \le j \le s.$ Now, we apply Corollary \ref{fib} to the  Riemann surfaces $C_1, \ldots, C_s$ to guarantee the existence of a Riemann surface $C$ such that $$JC \sim \Pi_{j=1}^sJC_j \times P \sim \Pi_{i=1}^tE_i \times P$$for a suitable abelian subvariety $P$ of $JC.$

\s

Assume $t$ to be odd; say $t=2s+1$ for some $s \ge 1.$ Let $C_j$ be a Riemann surface such that $$JC_j \sim E_j \times E_{j+s}$$for each $1 \le j \le s.$ Now, we apply Corollary \ref{fib} to the  Riemann surfaces $C_1, \ldots, C_s, E_t$ to guarantee the existence of a Riemann surface $C$ such that $$JC \sim \Pi_{j=1}^sJC_j \times E_t \times P \sim \Pi_{i=1}^tE_i \times P$$for a suitable abelian subvariety $P$ of $JC.$

\s

The computation of the genera is straightforward, and the proof is done. \qed

\subsection{Proof of Theorem B}
\begin{lemm} \label{pedro}
Let $G$ be a finite group such that $G=H_1 \cup \ldots \cup H_t$ where the subgroups $H_i$ of $G$ satisfy $H_i \cap H_j=\{1\}$ for $i \neq j.$ Then
\begin{equation*}
\oplus_{i=1}^t|H_i| \rho_{H_i} \cong (t-1)\rho_{\tiny \mbox{reg}} \oplus |G| W_1
\end{equation*}
\end{lemm}

\begin{proof}
Following \cite[p. 30]{Serre}, the character of the representation $|H_i| \rho_{H_i}$ at $g$ in $G$ is
\begin{displaymath}
 \left\{ \begin{array}{ll}
\,\,\,|G|  & \textrm{for $g=1$}\\
\# \mathscr{C}_{g,i}      & \textrm{for $g\neq 1$}
  \end{array} \right.
\end{displaymath}where $\mathscr{C}_{g, i}= \{s \in G : s^{-1}gs \in H_i \}$. Now, as $G$ is partitioned into its subgroups $H_i,$ we see that for each $g \neq 1$ the sets $\mathscr{C}_{g,i}$ and $\mathscr{C}_{g,j}$ are disjoint for $i \neq j$. It follows that the character of $ \oplus_{i=1}^t|H_i| \rho_{H_i}$ is
\begin{displaymath}
 \left\{ \begin{array}{ll}
 t|G|  & \textrm{for $g=1$}\\
\, |G|   & \textrm{for $g\neq 1$}
  \end{array} \right.
\end{displaymath}and, as the character of the regular representation is $|G|$ if $g=1$ and zero otherwise, the result follows directly by comparison of characters.
\end{proof}

\s

{\it Proof of Theorem B.}
Let us assume
\begin{equation} \label{juan}
JC \sim_G JC_G \times B_2^{n_2} \times \cdots \times B_r^{n_r}
\end{equation}
to be the group algebra decomposition of $JC$ with respect to $G$, and
$$
JC_{H_i} \sim JC_G \times B_2^{{n_2}^{H_i}} \times \cdots \times B_r^{{n_r}^{H_i}} \, ,
$$the induced  isogeny decomposition of $JC_{H_i}$ where $n_l^{H_i}=d_{V_l}^{H_i}/s_{V_l}$. Thus
$$
 JC_{H_1}^{|H_1|} \times \cdots \times JC_{H_t}^{|H_t|} \sim JC_G^{\sum_{i=1}^t |H_i|} \times B_2^{\sum_{i=1}^t n_2^{H_i}|H_i|} \times \cdots \times B_r^{\sum_{i=1}^t n_r^{H_i}|H_i|} \, .
$$

The fact that $|G| +t -1=\sum_{i=1}^t |H_i|$  implies that
$$
JC_{H_1}^{|H_1|} \times \cdots \times JC_{H_t}^{|H_t|}  \sim JC_G^{|G|} \times ( JC_G^{t-1} \times B_2^{\sum_{i=1}^t n_2^{H_i}|H_i|} \times \cdots \times B_r^{\sum_{i=1}^t n_r^{H_i}|H_i|} )
$$
and therefore we only need to prove that
$$
JC^{t-1} \sim JC_G^{t-1} \times B_2^{\sum_{i=1}^t n_2^{H_i}|H_i|} \times \cdots \times B_r^{\sum_{i=1}^t n_r^{H_i}|H_i|} \, .
$$

Again, after considering \eqref{juan}, it is enough to prove that $$
B_2^{(t-1)n_2} \times \cdots \times B_r^{(t-1)n_r} = B_2^{\sum_{i=1}^t n_2^{H_i}|H_i|} \times \cdots \times B_r^{\sum_{i=1}^t n_r^{H_i}|H_i|} \, .
$$
This equality is obtained if
$$
\sum_{i=1}^t \left[ n_l^{H_i}|H_i|-(t-1)n_l\right] = \frac{1}{s_{V_l}} \left[\sum_{i=1}^t   d_{V_l}^{H_i}|H_i|-(t-1)d_{V_l} \right]
$$
equals zero, for every $2 \le l \le r$.

Finally,  by Frobenius Reciprocity theorem, the previous expression can be rewritten as
$$
\frac{1}{s_{V_l}} \langle \oplus_{i=1}^t |H_i| \rho_{H_i}-(t-1) \rho_{\tiny \textup{reg}}, V_l \rangle_G
$$
and the result follows from Lemma \ref{pedro}. The proof of Theorem B is done. \qed

\section{Example}

 Let $q$ be an odd prime number. We consider a three-dimensional family $\mathscr{F}$ of compact Riemann surfaces $C$ of genus $g_C=4q-1$ admitting the action of a group of automorphisms $G$ isomorphic to the dihedral group$$\mathbf{D}_{2q}=\langle r,s : r^{2q}=s^2=(sr)^2=1 \rangle$$of order $4q,$ in such a way that:
 \begin{enumerate}
 \item the quotient $C_G$ has genus zero, and
 \item the associated $4q$-fold covering map $C \to C_G$ ramifies over six values; two values having $2q$ preimages (each with stabilizer a conjugate of the subgroup generated by $s$), two values having $2q$ preimages (each with stabilizer a conjugate of the subgroup generated by $sr$),  and two values having $2$ preimages (each with stabilizer the subgroup generated by $r$).
 \end{enumerate}

The existence of the family $\mathscr{F}$ is guaranteed by  Riemann's existence theorem (see, for example, \cite[Proposition 2.1]{brou}).

\s

Consider the subgroups $$H_1=\langle s \rangle, \hspace{0.4 cm} H_2=\langle sr \rangle \hspace{0.4 cm} \mbox{and} \hspace{0.4 cm} H_3=\langle r \rangle $$of $G.$ We remark that Theorem C cannot be applied in this case, since the hypothesis $(1)$ is not satisfied ($H_1$ and $H_2$ do not permute).

\s

{\bf Claim.} The collection $\{H_1, H_2, H_3\}$ is $G$-admissible.

\s

To prove the claim it is enough to consider a maximal collection of non-Galois associated (and non-trivial) complex irreducible representations of $G$. Namely, three representations of degree one
$$V_2: r \mapsto 1, s \mapsto -1, \,\,\, V_3: r \mapsto -1, s \mapsto 1, \,\,\, V_4: r \mapsto -1, s  \mapsto - 1 $$
and two of degree two \begin{equation*}
V_5 : r \mapsto
\left( \begin{array}{cc}
\omega_{2q} & 0 \\
0 & \overline{{\omega}_{2q}} \\
\end{array} \right) \,\,\,\,\, s \mapsto \left( \begin{array}{cc}
0 & 1 \\
1 & 0 \\
\end{array} \right)
\end{equation*}
\begin{equation*}
V_6 : r \mapsto
\left( \begin{array}{cc}
\omega_{q} & 0 \\
0 & \overline{{\omega}_{q}} \\
\end{array} \right) \,\,\,\,\, s \mapsto \left( \begin{array}{cc}
0 & 1 \\
1 & 0 \\
\end{array} \right)
\end{equation*}where $\omega_t:=\mbox{exp}(2 \pi i /t ).$

The claim follows after considering the following table, which summarizes the vector subspaces of $V_2, \ldots, V_6$ fixed under each of the subgroups $H_1, H_2, H_3$.

\s

\begin{center}
\begin{tabular}{|c|c|c|c|c|c|c|}  \hline
\, &  $V_2$ & $V_3$  & $V_4$ & $V_5$  & $V_6$  \\ \hline
$H_1$  & $\{0\}$ & $\langle 1 \rangle$ & $\{0\}$ & $\langle (1,1) \rangle$ & $\langle (1,1) \rangle$  \\ \hline
$H_2$  &  $\{0\}$ & $\{0\}$ & $\langle 1 \rangle$ & $\langle (1, \omega_{2q}) \rangle$ & $\langle (1, \omega_{q}) \rangle$  \\ \hline
$H_3$  & $\langle 1 \rangle$ & $\{0\}$ & $\{0\}$ & $\{0\}$ & $\{0\}$  \\ \hline
\end{tabular}
\end{center}

\s

\s

Hence, by Theorem \ref{plano}, the claim above implies that \begin{equation*}  JC \sim JC_{H_1} \times JC_{H_2} \times JC_{H_3} \times P \end{equation*}for a suitable abelian subvariety $P$ of $JC.$

Furthermore, it is not difficult to compute that $$g_{C_{H_1}}=2q-1, \hspace{0.5 cm} g_{C_{H_2}}=2q-1 \hspace{0.5 cm} \mbox{and} \hspace{0.5 cm} g_{C_{H_3}}=1 \, ,
$$
showing that their sum agrees with $g_C$.  Thereby $P=0$ and the decomposition \begin{equation} \label{deco1} JC \sim JC_{H_1} \times JC_{H_2} \times JC_{H_3} \end{equation}is obtained.

\s

We also mention some facts related to this example:

\s

\begin{enumerate}
\item by Proposition \ref{equiv}, we have that $\oplus_{i=1}^3 \rho_{H_i}  \cong \rho_{\tiny \textup{reg}}  \oplus 2 W_1.$
\s

\item By Corollary \ref{rubi}, the Prym varieties associated to the subgroups $H_1$ and $H_2$ --whose dimension is $2q$-- contain an elliptic curve, for all $q$.

\s

\item The group algebra decomposition of $JC$ with respect to $G$ is \begin{equation} \label{deco2}JC \sim_G B_1 \times B_2 \times B_3 \times B_4 \times B_5^2 \times B_6^2 \ , \end{equation}

by \cite[Theorem 5.12]{yoibero}, the dimensions of the factors are  \begin{displaymath}
\mbox{dim}(B_l)= \left\{ \begin{array}{ll}
 \,\,\,\,\,0 & \textrm{if $l=1$}\\
 \,\,\,\,\,1 & \textrm{if $l=2,3,4$}\\
q-1 & \textrm{if $l=5,6$}
  \end{array} \right.
\end{displaymath}and the decompositions \eqref{deco1} and \eqref{deco2} are related by the isogenies
$$   JC_{H_1} \sim B_3 \times B_5 \times B_6 \ \ , \ \   JC_{H_2} \sim B_4 \times B_5 \times B_6  \ \ , \ \    JC_{H_3} \sim B_2 \, . $$

\s

\item  The collection $\{H_1, H_3\}$ satisfies  hypotheses $(1)$ and $(2)$ of Theorem C, but it does not satisfy  hypothesis $(3)$. Indeed, $$g_{C_{H_1}}+g_{C_{H_3}} = 2q \neq 4q-1= g_C.$$ Thus, Theorem C cannot be applied. In contrast, it is straightforward to see that $\{H_1, H_3\}$ is, in fact, a $G$-admissible collection. Then Theorem \ref{plano} asserts that $$JC \sim JC_{H_1} \times JC_{H_3} \times Q$$for some abelian subvariety $Q$ of $JC$ of dimension $2q-1.$

\s

\item The subgroup $H_1$ and $H_4=\langle r^q \rangle$ permute, but the genus of $C_{H_1H_4}$ is positive. Then hypothesis $(2)$ in Theorem C is not satisfied, and this result cannot be applied. On the other hand, the vector subspaces of $V_2, \ldots, V_6$ fixed under $H_1$ and $H_4$ are

\s

\begin{center}
\begin{tabular}{|c|c|c|c|c|c|c|}  \hline
\, &  $V_2$ & $V_3$  & $V_4$ & $V_5$  & $V_6$  \\ \hline
$H_1$  & $\{0\}$ & $\langle 1 \rangle$ & $\{0\}$ & $\langle (1,1) \rangle$ & $\langle (1,1) \rangle$  \\ \hline
$H_4$  & $\langle 1 \rangle$ & $\{0\}$ & $\{0\}$ & $\{0\}$ & $\langle (1,1) \rangle$  \\ \hline\end{tabular}
\end{center}
\s
showing that the collection $\{H_1, H_4\}$ is $G$-admissible. Hence, by Theorem \ref{plano}, we have that $$JC \sim JC_{H_1} \times JC_{H_4} \times E$$where $E$ is an elliptic curve.
\item Note that the collection $\{H_1, H_4\}$ is $G$-admissible but it is not $\langle H_1, H_4 \rangle$-admissible.
\end{enumerate}
\s


\begin{thebibliography}{999}
\bibitem{bl}
{\sc Ch. Birkenhake and H. Lange},
 { \em Complex Abelian Varieties,} $2^{nd}$ edition,
Grundl. Math. Wiss. {\bf 302}, Springer, 2004.

\bibitem{brou}
{\sc S. A. Broughton}, {\em Classifying finite group actions on surfaces of low genus.} J. Pure Appl. Algebra {\bf 69} (1991), no. 3, 233--270.

\bibitem{BF}
{\sc N. Bruin and E. Flynn,}
{\em $n$-covers of hyperelliptic curves,}
Math. Proc. Cambridge Philos. Soc. {\bf 134} (2003), no. 3, 397--405.

\bibitem{cr}
{\sc A. Carocca and R. E. Rodr\'iguez,}
{\em Jacobians with group actions and rational idempotents.}
J. Algebra \textbf{306} (2006), no. 2, 322--343.

\bibitem{CHQ}
{\sc M. Carvacho, R. A. Hidalgo and S. Quispe,}
{\em Jacobian variety of generalized Fermat curves,}
Q. J. Math. {\bf 67} (2016), no. 2, 261--284.

\bibitem{Earle}
{\sc C. J. Earle,}
{\em Some Jacobian varieties which split,}
Lectures Notes in Mathematics {\bf 747} (1979), 101--107.

\bibitem{EJPS}
{\sc T. Ekedahl and J. P. Serre,}
{\em Exemples de courbes alg\'ebriques \`a jacobienne compl\`etement d\'ecomposable,}
C. R. Math. Acad. Sci. Paris S\'er. I {\bf 317}, vol. 5 (1993) 509--513.



\bibitem{F}
{\sc W. Fulton and J. Hansen,}
{\em A connectedness theorem for projective varieties, with applications to intersections and singularities of mappings,} Annals of Math. {\bf 110} (1979), 159--166.

\bibitem{GP}
{\sc D. Glass and R. Pries,}
{\em Hyperelliptic curves with prescribed $p$-torsion,} Manuscripta Math. {\bf 11}7 (2005), no. 3, 299--317.

\bibitem{jap}
{\sc T. Hayashida and M. Nishi,}
{\em Existence of curves of genus two on a product of two elliptic curves,}
Math. Soc. Japan {\bf 17} (1965), no. 1, 1--16.

\bibitem{Ruben}
{\sc R. A. Hidalgo,}
{ \em Elliptic factors on the Jacobian variety of Riemann surfaces}, arXiv: 1507.07822.

\bibitem{RA}
{\sc R. A. Hidalgo,} {\em The fiber product of Riemann surfaces: A Kleinian group point of view,} Analysis and Mathematical Physics {\bf 1} (2011), 37--45

\bibitem{nos}
{\sc R. A. Hidalgo, L. Jim\'enez, S. Quispe and S. Reyes-Carocca,} {\em Quasiplatonic curves with symmetry group $\mathbb{Z}_2^2 \rtimes \mathbb{Z}_m$ are definable over $\mathbb{Q}$,}  Bull. London Math. Soc. {\bf 49} (2017) 165--183.


\bibitem{RAS}
{\sc R. A. Hidalgo, S. Reyes-Carocca and A. Vega,} {\em On the fiber product of Riemann surfaces II}, arXiv:1611.07880


\bibitem{RuRu}
{\sc R. A. Hidalgo and R. E. Rodr{\'\i}guez,} {\em A remark on the decomposition of the Jacobian variety of Fermat curves of prime degree}, Arch. Math. {\bf 105} (2015), 333--341.


\bibitem{KR}
{\sc E. Kani and M. Rosen,} {\em Idempotent relations and factors of Jacobians},
Math. Ann. {\bf 284} (1989), 307--327.

\bibitem{l-r}
{\sc H. Lange and S. Recillas,}
{ \em Abelian varieties with group actions}.
J. Reine Angew. Mathematik \textbf{575} (2004), 135--155.

\bibitem{RM}
{\sc{R. Miranda,}} {\em{Algebraic curves and Riemann surfaces,}} Graduate Studies in Maths. {\bf 5}, Amer. Math. Soc. (1995).

\bibitem{P}
{\sc J. Paulhus,}
{ \em Decomposing Jacobians of curves with extra automorphisms}, Acta Arith. {\bf 132} (2008), no. 3, 231--244.

\bibitem{PA}
{\sc J. Paulhus and A. M. Rojas,}
{ \em Completely decomposable Jacobian varieties in new genera},  Experimental Mathematics {\bf 26} (2017), no. 4, 430--445.

\bibitem{RR}
{\sc S. Recillas and R. E.  Rodr{\'\i}guez,}
{ \em Prym varieties and fourfold covers. II. The dihedral case,}
Contemp. Math. {\bf 397} (2006), 177–-191.

\bibitem{RR2}
{\sc G. Riera and R. E.  Rodr{\'\i}guez,}
{ \em Uniformization of surfaces of genus two with automorphisms}, Math. Ann. {\bf 282} (1988), no. 1, 51-–67.

\bibitem{yoibero}
{\sc A. M.   Rojas}, {\em Group actions on Jacobian varieties}, Rev. Mat. Iber. {\bf 23} (2007), 397--420.

\bibitem{Serre}
{\sc J. P. Serre}, {\em Linear Representations of Finite Groups,} Graduate Texts in Maths {\bf 42}, Springer.
\end{thebibliography}
\end{document}